\pdfoutput=1
\documentclass[11pt]{article}
\usepackage[letterpaper,margin=1in]{geometry}
\usepackage{lmodern}
\usepackage{microtype}
\usepackage[T1]{fontenc}
\usepackage{amsmath,amssymb,amsthm,mathtools,bm}
\usepackage{booktabs,array,graphicx}
\usepackage{tikz}
\usetikzlibrary{arrows.meta,positioning,calc}
\usepackage{cite}
\usepackage[hidelinks]{hyperref}
\hypersetup{pdftitle={The Complexity of Minimal Positive Realization},pdfauthor={Zhaobo Liu},pdfsubject={Minimal positive realization}}
\usepackage{url}
\newtheorem{theorem}{Theorem}
\newtheorem{lemma}[theorem]{Lemma}

\newtheorem{corollary}[theorem]{Corollary}
\theoremstyle{definition}
\newtheorem{definition}{Definition}
\theoremstyle{remark}

\newcommand{\R}{\mathbb R}
\newcommand{\Q}{\mathbb Q}

\newcommand{\ER}{\exists\mathbb R}
\newcommand{\dplus}{d_+}
\newcommand{\mcm}{d_{\mathrm{McM}}}
\newcommand{\PR}{\textnormal{\textsc{PosRealOrder}}}
\newcommand{\Inv}{\textnormal{\textsc{ETR-INV}}}
\newcommand{\diag}{\operatorname{diag}}

\newcommand{\e}{\mathbf e}

\newcommand{\coeffge}{\succeq_{\mathrm{coef}}}

\title{The Complexity of Minimal Positive Realization}
\author{Zhaobo~Liu%
\thanks{This work was supported in part by the National Natural Science Foundation of China under Grant 12401585, the Guangdong Basic and Applied Basic Research Foundation under Grant 2024A1515011542, and the General Program of Shenzhen Natural Science Foundation under Grant JCYJ20250604181037012.}
\thanks{Zhaobo Liu is with the Institute for Advanced Study, Shenzhen University, 3883 Baishi Road, Nanshan District, Shenzhen 518060, China (e-mail: liuzhaobo@szu.edu.cn).}
}
\date{}
\begin{document}
\maketitle
\begin{abstract}
We study the complexity of minimal positive realization for discrete-time systems. We prove that, given a strictly proper transfer matrix with rational coefficients and an order bound encoded in unary, deciding whether a positive realization exists within that bound is complete for the existential theory of the reals. We further prove hardness for stable single-input single-output transfer functions that admit positive realizations, even when the bound equals the McMillan degree.
\end{abstract}

\section{Introduction}\label{sec:intro}
Positive linear systems model processes involving nonnegative quantities, such as populations, chemical concentrations, and material inventories \cite{FarinaRinaldi2000,BenvenutiFarina2004}. Their states and outputs remain nonnegative under nonnegative initial conditions and inputs. For discrete-time linear systems, the positive realization problem is to find a state-space model with nonnegative system matrices that realizes a given transfer matrix exactly. The minimal positive realization problem seeks the smallest number of states among all such models, preserving both the prescribed input--output behavior and nonnegativity. This problem has been studied for several decades \cite{BenvenutiFarina1999,Benvenuti2022}. Nonnegativity can force the minimum positive order above the McMillan degree, the minimum number of states in an ordinary realization. We study the computational difficulty of determining how many states are required.

Determining the minimum positive order requires information about both existence and the number of states needed. Invariant polyhedral cones characterize whether a positive realization exists \cite{VanDenHof1997}. If a positive realization exists, lower bounds exclude smaller orders \cite{Hadjicostis1999}, while constructive algorithms under assumptions on the dominant pole supply upper bounds \cite{CzajaJamingMatolcsi2008}. Matching a lower bound with a construction proves minimality. Such bounds can be obtained when a nonnegative impulse response has a zero sample and is strictly positive at all later times \cite{NagyMatolcsi2003Lowerbound}. For particular classes, the McMillan degree itself can be tested as a candidate minimum. Constructions identify cases where this degree is attained for specific pole and residue configurations \cite{HalmschlagerMatolcsi2005}, and necessary and sufficient conditions are available for third-order transfer functions with distinct real poles \cite{Benvenuti2013RealPoles}. These results determine the minimum order when a construction attains a lower bound or a criterion confirms that the McMillan degree is attainable.

Smaller positive models can also be sought by reducing a given model or by restricting the realization form. For exact reduction of a positive model, recent work gives conditions for using nonnegative reduction matrices \cite{Cortese2025}. For single-input single-output transfer functions, feasibility of a positive realization in Markov form at a prescribed order can be determined by linear programming \cite{TaghavianSjolund2025}. We ask how difficult it is to decide whether a given transfer matrix admits a positive realization within a prescribed order bound.

Related algebraic feasibility problems have known complexity classifications. Spectral conditions connect positive realization to variants of the nonnegative inverse eigenvalue problem \cite{Benvenuti2020NIEP}, whose real variant is NP-hard \cite{BorobiaCanogar2017}. Exact nonnegative matrix factorization is NP-hard \cite{Vavasis2009} and, by Shitov's result, complete for the existential theory of the reals \cite[Theorem~2]{Shitov2018}. We use real polynomial feasibility to classify the positive realization order problem.

We prove that, given a strictly proper transfer matrix with rational coefficients and an order bound encoded in unary, deciding whether a positive realization exists within that bound is complete for the existential theory of the reals. The problem is therefore polynomial-time equivalent to deciding whether finite systems of polynomial equalities and inequalities with rational coefficients have real solutions. Hardness already holds for stable single-input single-output (SISO) transfer functions when the bound equals the McMillan degree $m$. Every transfer function produced by the reduction has a positive realization, whether or not it has one with $m$ states. Thus, even when positive realizability is guaranteed, deciding whether nonnegativity requires additional states beyond an ordinary minimal realization is computationally hard.

For the constructed transfer functions, initial zero impulse-response samples and nonnegativity force every positive realization of order $m$ to contain a chain from input to output through all $m$ states. This chain structure allows us to express the existence of an $m$-state positive realization through polynomial coefficient inequalities for the denominators used in the construction. We then encode bounded real equations \cite[Definition~5 and Theorem~7]{Abrahamsen2018} in these inequalities so that the equations have a real solution exactly when the constructed transfer function admits an $m$-state positive realization. This equivalence transfers the computational hardness of the bounded equations to positive realization at order $m$.

Section~\ref{sec:problem} states the problem and main results. Sections~\ref{sec:structure} and~\ref{sec:construction} develop the structural criteria and the transfer functions for the reduction. Section~\ref{sec:bits} proves Theorem~\ref{thm:main}. Appendices~\ref{app:structure} and~\ref{app:encoding} contain the structural proofs and the encoding of bounded equations.

\noindent\textit{Notation:} The symbols $\R$ and $\Q$ denote the sets of real and rational numbers, respectively. All matrix and vector inequalities below are entrywise. We write $\e_i$ for the $i$th standard basis vector and $\R_+^m$ for the nonnegative orthant in $\R^m$. For a polynomial $f(t)$, $[t^k]f$ denotes the coefficient of $t^k$. For real polynomials $f$ and $g$, $f\coeffge g$ means that every coefficient of $f-g$ is nonnegative.

\section{Problem Formulation and Main Results}\label{sec:problem}
We define the minimum positive order and its associated decision problem.

\begin{definition}\label{def:positive}
Let $H$ be a nonzero strictly proper real rational transfer matrix with $p$ inputs and $q$ outputs. For a positive integer $n$, a positive realization of order $n$ is a triple $(A,B,C)$ with $A\in\R^{n\times n}$, $B\in\R^{n\times p}$, and $C\in\R^{q\times n}$ satisfying
\begin{equation}\label{eq:realization}
 H(z)=C(zI-A)^{-1}B,\qquad A,B,C\ge0.
\end{equation}
The minimum positive order is
\begin{equation}\label{eq:order}
 \dplus(H)=\min\{n:H\text{ has a positive }n\text{-state realization}\},
\end{equation}
with value $+\infty$ if no finite positive realization exists.
\end{definition}

The Markov parameters are defined by
\begin{equation}\label{eq:markov}
 H(z)=\sum_{k\ge0}H_kz^{-(k+1)},\qquad H_k=CA^kB.
\end{equation}
In the SISO case, we write a realization as $(A,b,c)$ with $B=b$, $C=c^\top$, and $b,c\in\R^n$, and denote the scalar Markov parameters by $h_k=c^\top A^kb$.

\begin{definition}\label{def:mcm}
The McMillan degree $\mcm(H)$ is the minimum number of states among all real state-space realizations of $H$ \cite[Sec.~II]{BenvenutiFarina2004}.
\end{definition}

Write $m=\mcm(H)$. Every positive realization is also a real state-space realization, so $\dplus(H)\ge m$. For a reduced SISO transfer function $H=N/D$ with monic denominator $D$, $m=\deg D$.

\subsection{Order decision problem}
Given a transfer matrix $H$ and an order bound $n$, we ask whether $H$ has a positive realization with at most $n$ states. This is the decision version of the minimal positive realization problem \cite{Benvenuti2022}. Computing the minimum positive order would solve this decision problem for every $n$.

\begin{definition}\label{def:orderdecision}
Let $H$ be a nonzero strictly proper transfer matrix with rational coefficients and let $n$ be a positive integer. The positive realization order problem, denoted by $\PR$, is to determine whether
\begin{equation}\label{eq:language}
 \dplus(H)\le n.
\end{equation}
\end{definition}

For the complexity analysis, we use rational data to obtain a finite exact representation. Each rational number is encoded by its integer numerator and denominator in binary. Each entry of $H$ is specified by the full coefficient lists of its numerator and denominator polynomials. The order bound $n$ is encoded in unary as $1^n$, a string of $n$ ones. The input size is the total number of bits in these data.

Real polynomial feasibility asks whether a finite system of polynomial equalities and inequalities with rational coefficients has a real solution. A polynomial-time many-one reduction transforms each instance of one decision problem into an instance of another in polynomial time, preserving the answer. The class $\ER$ consists of decision problems that admit such a reduction to real polynomial feasibility. A problem is $\ER$-hard if every problem in $\ER$ reduces to it, and $\ER$-complete if it is both $\ER$-hard and in $\ER$ \cite[Sec.~1.1]{SchaeferCardinalMiltzow2024}.

\subsection{Main results}

\begin{theorem}\label{thm:main}
The decision problem $\PR$ is $\ER$-complete. Hardness already holds for stable SISO transfer functions that admit a positive realization, with order bound $n=\mcm(H)$.
\end{theorem}

The hardness assertion in Theorem~\ref{thm:main} concerns the number of states needed when a positive realization exists. In particular, it establishes hardness of deciding whether nonnegativity requires more states than an ordinary minimal realization.

For SISO systems, the result also describes the difficulty of finding a change of state coordinates that makes a minimal realization positive.

\begin{corollary}\label{cor:coordinates}
Given a minimal SISO realization $(A_0,b_0,c_0)$ with rational entries, deciding whether there is an invertible real matrix $T$ satisfying
\begin{equation}\label{eq:coordinates}
 T^{-1}A_0T\ge0,\qquad T^{-1}b_0\ge0,\qquad c_0^\top T\ge0
\end{equation}
is $\ER$-complete.
\end{corollary}
\begin{proof}
Any two minimal realizations of the same transfer function are related by an invertible similarity transformation \cite[Sec.~II]{BenvenutiFarina2004}. Thus~\eqref{eq:coordinates} holds exactly when a positive realization exists at the McMillan degree. A minimal rational companion realization of each transfer function constructed in the proof of Theorem~\ref{thm:main} is computable in polynomial time, which transfers hardness to~\eqref{eq:coordinates}. For membership in $\ER$, introduce matrices $T,V$ with $TV=VT=I$ and impose $VA_0T\ge0$, $Vb_0\ge0$, and $c_0^\top T\ge0$.
\end{proof}

The proof of Theorem~\ref{thm:main} is given in Section~\ref{sec:mainproof}.

\section{Positive Realization at the McMillan Degree}\label{sec:structure}

Consider a transfer function $H=1/p$, where $p$ is monic of degree $m$. We ask whether $H$ has a positive realization with exactly $m$ states. Its delayed impulse response and nonnegativity force a chain through all states, reducing the search to the remaining entries of a normalized state matrix. For the denominator coefficient pattern specified in Lemma~\ref{lem:compression}, we then express existence of that matrix as inequalities on polynomial coefficients. These inequalities provide the connection to real polynomial feasibility used in Section~\ref{sec:construction}.

\subsection{State structure imposed by the impulse response}\label{sec:hess}

Let $H(z)=1/p(z)$, where $p$ is a monic real polynomial of degree $m$. Since the numerator is one, the McMillan degree is $m$. The leading term of $p$ is $z^m$, so $H(z)=z^{-m}+O(z^{-m-1})$ as $z\to\infty$. Comparing this expansion with~\eqref{eq:markov} gives
\begin{equation}\label{eq:prefix}
 h_k=0\quad\text{for }0\le k<m-1,\qquad h_{m-1}=1.
\end{equation}

Consider a positive realization with $m$ states. The input gain at state $j$ is $b_j$, the output gain from state $i$ is $c_i$, and $A_{ij}$ propagates state $j$ to state $i$ in one time step. For example, $h_1=\sum_{i,j}c_iA_{ij}b_j$ sums the contributions of all routes with one state transition. More generally, $h_k=c^\top A^kb$ sums the contributions of routes with $k$ state transitions. All these contributions are nonnegative. Hence~\eqref{eq:prefix} excludes every positive route with fewer than $m-1$ transitions and requires at least one with exactly $m-1$.

Such a shortest route cannot repeat a state. Otherwise, deleting the part between two occurrences of the same state would give a shorter positive route, contradicting~\eqref{eq:prefix}. It therefore passes through $m$ distinct states and uses all states of the realization. Number them along the route as $1,\ldots,m$. Input acting on a later state, or output read from an earlier state, would produce a shorter route. A forward coupling that skips a state would do the same. Thus the input acts only on state 1, the output reads only state $m$, and the remaining possible couplings are diagonal or directed from later states to earlier states. For $m=3$, the conditions $h_0=h_1=0$ and $h_2=1$ therefore give the chain $1\to2\to3$, with the input acting only on state 1 and the output taken only from state 3.

A positive diagonal scaling makes the input gain and every forward coupling equal to one. Since their product with the output gain is $h_{m-1}=1$, the output gain also becomes one. The following lemma states the resulting existence criterion.

\begin{lemma}\label{lem:hess}
Let $p$ be a monic real polynomial of degree $m\ge1$. Then $\dplus(1/p)=m$ if and only if there exists $J\in\R^{m\times m}$ with $J\ge0$ such that
\begin{equation}\label{eq:hess}
\begin{aligned}
 J_{i+1,i}&=1,\quad J_{ij}=0\ \text{for }i>j+1,\\
 \det(zI-J)&=p(z).
\end{aligned}
\end{equation}
Every $m$-state positive realization of $1/p$ can be transformed to $(J,\e_1,\e_m)$ by a state permutation and a positive diagonal similarity.
\end{lemma}
The matrix in~\eqref{eq:hess} is upper Hessenberg with unit subdiagonal. For input vector $\e_1$ and output vector $\e_m$, its transfer function is $\e_m^\top(zI-J)^{-1}\e_1=1/\det(zI-J)$. Hence $1/p$ has a positive realization of order $m$ exactly when the nonnegative entries on and above the diagonal can be chosen to satisfy $\det(zI-J)=p(z)$. The proof of Lemma~\ref{lem:hess} is given in Appendix~\ref{app:structure}.

\subsection{Conditions on the denominator coefficients}\label{sec:compression}

It remains to decide whether the free nonnegative entries of $J$ can reproduce the prescribed denominator $p$. To relate those entries to the denominator coefficients, write
\begin{equation}\label{eq:reverse}
 q(t)=t^m p(1/t)=1+\sum_{k=1}^m q_kt^k.
\end{equation}
For a matrix $J$ satisfying~\eqref{eq:hess}, $q(t)=\det(I-tJ)$. A return coupling $J_{ij}$, with $i\le j$, closes a cycle along the forward chain from state $i$ to state $j$. The cycle has length $j-i+1$ and weight $J_{ij}$, since all forward weights are one. In the determinant expansion, a single cycle contributes a negative term and two cycles sharing no state contribute a positive product \cite[Eq.~(5)]{TorreMayo2007}. The full expansion is given in Appendix~\ref{app:structure}.

In the five-state model in Fig.~\ref{fig:stateexample}(a), the return couplings $J_{12},J_{23},J_{35}$ have weight one, and all other entries on and above the diagonal are zero. The two length-two cycles contribute $-2t^2$, the length-three cycle contributes $-t^3$, and the only disjoint pair contributes $t^5$. Hence $q(t)=1-2t^2-t^3+t^5$. The equivalent model in Fig.~\ref{fig:stateexample}(b) is obtained by moving the coupling $J_{23}$ to $J_{12}$, making its weight two, and adding $J_{15}=1$. Both models therefore have the same denominator $p(z)=z^5-2z^3-z^2+1$ and, by Lemma~\ref{lem:hess}, the same transfer function $1/p$.

\begin{figure}[t]
\centering
\begin{tikzpicture}[x=1.15cm,y=1cm,>=Latex,font=\footnotesize,
 state/.style={circle,draw,minimum size=5mm,inner sep=0pt}]
\foreach \row/\yy in {a/0,b/-2.4}{
  \foreach \i in {1,...,5}{\node[state] (\row\i) at ({\i-1},\yy) {$\i$};}
  \foreach \i/\j in {1/2,2/3,3/4,4/5}{\draw[->] (\row\i)--node[below] {$1$}(\row\j);}
  \draw[->] (-.8,\yy) node[left] {$u$}--(\row1);
  \draw[->] (\row5)--(4.8,\yy) node[right] {$y$};
}
\node[anchor=west] at (-.8,1.25) {(a) Original couplings};
\draw[->] (a2.north)..controls (1,.65) and (0,.65)..node[above] {$1$}(a1.north);
\draw[->] (a3.north)..controls (2,.65) and (1,.65)..node[above] {$1$}(a2.north);
\draw[->] (a5.north)..controls (4,.9) and (2,.9)..node[above] {$1$}(a3.north);
\node[anchor=west] at (-.8,-1.15) {(b) Redistributed couplings};
\draw[->] (b2.north)..controls (1,-1.75) and (0,-1.75)..node[above] {$2$}(b1.north);
\draw[->] (b5.north)..controls (4,-1.5) and (2,-1.5)..node[above] {$1$}(b3.north);
\draw[->] (b5.south)..controls (4,-3.15) and (0,-3.15)..node[below] {$1$}(b1.south);
\end{tikzpicture}
\caption{Two five-state positive realizations of $1/(z^5-2z^3-z^2+1)$. All forward couplings have weight one. The input enters state 1, the output is taken at state 5, and return couplings carry the indicated weights.}
\label{fig:stateexample}
\end{figure}
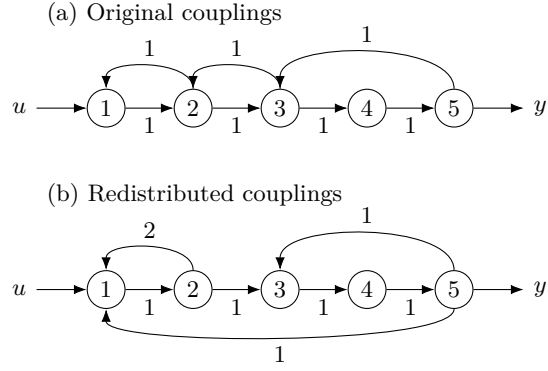

The coefficient pattern used in the reduction allows redistribution of the return couplings to be expressed by a small number of weights. For an integer $g\ge2$, the conditions $q_1=\cdots=q_{g-1}=0$ exclude nonzero cycles of length less than $g$. If such a cycle existed, take the shortest such length. Its coefficient would be minus the sum of the weights of cycles of that length, since products of two or more nonzero cycles have greater total length. This contradicts the zero coefficient. With $m=3g-1$, at most two nonzero cycles can occupy disjoint state sets. The denominator coefficients therefore depend only on weights of individual cycles and products of two weights. The next lemma replaces the unknown locations of the couplings by two families of nonnegative coefficients.

\begin{lemma}\label{lem:compression}
Let $g\ge2$ be an integer and set $m=3g-1$. Let $q(t)=1+\sum_{k=1}^m q_kt^k$ have real coefficients and satisfy
\begin{equation}\label{eq:short}
 q_k=0\ \text{for }1\le k<g,\qquad q_k=-s_k\ \text{for }g\le k<2g,
\end{equation}
where $s_k\ge0$. Set $p(z)=z^m q(1/z)$. Then $\dplus(1/p)=m$ if and only if there exist $X_k,Y_k\ge0$ satisfying
\begin{align}\label{eq:splitineq}
 X_k+Y_k&=s_k &&\text{for }g\le k<2g,\nonumber\\
 [t^j]X(t)Y(t)&\ge q_j &&\text{for }2g\le j\le m,
\end{align}
where $X(t)=\sum_{k=g}^{2g-1}X_kt^k$ and $Y(t)=\sum_{k=g}^{2g-1}Y_kt^k$.
\end{lemma}
Lemma~\ref{lem:compression} replaces the unknown entries of $J$ by the nonnegative coefficients of two polynomials, $X$ and $Y$, subject to~\eqref{eq:splitineq}. The complete proof is given in Appendix~\ref{app:structure}.

In Fig.~\ref{fig:stateexample}(b), the length-two cycle starting at state 1 gives $X(t)=2t^2$, and the length-three cycle ending at state 5 gives $Y(t)=t^3$. Their product contributes $2t^5$. The added length-five cycle subtracts $t^5$, leaving the prescribed coefficient $q_5=1$.

The equality $X_k+Y_k=s_k$ fixes the sum $X+Y$ from the denominator coefficients. Thus only the difference between $X$ and $Y$ remains to be determined. Define the known polynomial $S(t)=\sum_{k=g}^{2g-1}s_kt^k$ and the unknown polynomial $Z(t)=X(t)-Y(t)$, with coefficients $Z_k=X_k-Y_k$. Then $X=(S+Z)/2$ and $Y=(S-Z)/2$, so their coefficients are nonnegative exactly when $|Z_k|\le s_k$ for $g\le k<2g$. The product constraints can also be expressed in terms of $Z$. Since
\begin{equation}\label{eq:squareidentity}
 XY=\frac{S^2-Z^2}{4},
\end{equation}
the inequalities in~\eqref{eq:splitineq} are equivalent to
\begin{equation}\label{eq:coefficientcriterion}
 [t^j]Z^2\le[t^j]S^2-4q_j\qquad\text{for }2g\le j\le m.
\end{equation}
For the coefficient pattern in Lemma~\ref{lem:compression}, the search over positive state matrices is therefore equivalent to a search over the coefficients of $Z$. The denominator fixes both the bounds $|Z_k|\le s_k$ and the right-hand side of~\eqref{eq:coefficientcriterion}. Section~\ref{sec:construction} chooses these denominator coefficients so that an $m$-state positive realization exists exactly when a prescribed system of bounded real equations is feasible.

\section{Transfer Functions for the Hardness Reduction}\label{sec:construction}

The reduction encodes bounded real equations in the denominator of a SISO transfer function. The coefficient criterion from Section~\ref{sec:structure} makes feasibility equivalent to positive realization at the McMillan degree. The construction also guarantees positive realizability for every instance and permits a scaling that ensures stability.

\subsection{From bounded equations to coefficient constraints}\label{sec:source}

The problem $\Inv$ asks whether a finite system of equations of the forms $x=1$, $x+y=z$, and $xy=1$ has a real solution with every variable in $[1/2,2]$. This problem is $\ER$-complete \cite[Definition~5 and Theorem~7]{Abrahamsen2018}. Our reduction starts from $\Inv$. Replace $x=1$ by $x^2=1$, which is equivalent on $[1/2,2]$. Add a variable constrained to equal one, so every $\Inv$ instance has at least one variable.

To encode the $\Inv$ equations, the following auxiliary problem introduces new polynomials $S$ and $Z$. Section~\ref{sec:encoding} places their coefficients at the degrees required by Lemma~\ref{lem:compression}.

\begin{definition}\label{def:square}
For an integer $D\ge1$, the given data are a polynomial with rational coefficients
\begin{equation}\label{eq:amplitudes}
 S(t)=\sum_{e=0}^D s_et^e,\qquad s_0=1,\quad 0\le s_e\le1,\quad s_D>0,
\end{equation}
a set of degrees $\mathcal D\subseteq\{0,\ldots,2D\}$, and rational bounds $\beta_d$ satisfying
\begin{equation}\label{eq:range}
 -[t^d]S^2\le\beta_d\le[t^d]S^2\qquad\text{for }d\in\mathcal D.
\end{equation}
The coefficient feasibility problem is to determine whether there exist real coefficients $z_0,\ldots,z_D$ such that $Z(t)=\sum_{e=0}^D z_et^e$ satisfies
\begin{equation}\label{eq:squareproblem}
\begin{aligned}
 |z_e|&\le s_e &&\text{for }0\le e\le D,\\
 [t^d]Z^2&\le\beta_d &&\text{for }d\in\mathcal D.
\end{aligned}
\end{equation}
\end{definition}

Definition~\ref{def:square} isolates the coefficient constraints that will be placed in the denominator. Its polynomial $S$ supplies bounds on the unknown coefficients of $Z$. Section~\ref{sec:encoding} places the coefficients of $S$ at the degrees required by Lemma~\ref{lem:compression}. The constraints in~\eqref{eq:squareproblem} are quadratic inequalities in $z_0,\ldots,z_D$, since
\[
 [t^d]Z^2=\sum_{\substack{0\le i,j\le D\\i+j=d}}z_i z_j.
\]

\begin{lemma}\label{lem:squarehard}
Given an $\Inv$ instance, one can construct in polynomial time a polynomial $S$, a degree set $\mathcal D$, and bounds $\beta_d$ satisfying Definition~\ref{def:square}. The $\Inv$ instance has a real solution if and only if the resulting constraints~\eqref{eq:squareproblem} have a real solution. The degree $D$ is bounded by a polynomial in the encoding length of the $\Inv$ instance.
\end{lemma}

The coefficient feasibility problem in Definition~\ref{def:square} is therefore $\ER$-hard. The construction and proof of Lemma~\ref{lem:squarehard} are given in Appendix~\ref{app:encoding}.

\subsection{From coefficient constraints to positive realization}\label{sec:encoding}

Given $S,\mathcal D,\beta_d$ from Definition~\ref{def:square}, we construct a denominator $p$ so that $1/p$ has a positive realization of order $m$ exactly when the coefficient constraints are feasible. We place the encoded inequalities at even degrees and use the remaining odd coefficients to guarantee positive realizability at some finite order.

The identity~\eqref{eq:squareidentity} converts the bounds on coefficients of $Z^2$ into lower bounds on coefficients of $(S^2-Z^2)/4$. Collect these bounds in
\begin{equation}\label{eq:boundpoly}
 B(t)=\sum_{d\in\mathcal D}\frac{[t^d]S^2-\beta_d}{4}\,t^d,
 \qquad 0\preceq_{\rm coef}B\preceq_{\rm coef}\tfrac12S^2.
\end{equation}
Choose
\begin{equation}\label{eq:sizes}
 g=6D+4,\qquad m=3g-1=18D+11,
\end{equation}
and define
\begin{equation}\label{eq:ab}
 a(t)=t^g\bigl(1+tS(t^2)\bigr),\qquad b(t)=t^{2g+2}B(t^2).
\end{equation}
The coefficients of $a$ specify the sum $X+Y$ in Lemma~\ref{lem:compression}. They are one at degree $g$ and $s_e$ at degree $g+1+2e$ for $0\le e\le D$. Since $\deg a=g+2D+1<2g$, they lie in the required range. The coefficient of $B$ at degree $d$ becomes the coefficient of $b$ at the even degree $2g+2+2d<m$.

To ensure positive realizability, set $\gamma=1+a(1)b(1)$ and
\[
 E(t)=\gamma\sum_{\substack{2g+1\le j\le m\\j\text{ odd}}}t^j.
\]
The uniform odd coefficients of $E$ will give the inequality $aE\coeffge ab$ used in Lemma~\ref{lem:existence}. Define
\begin{equation}\label{eq:denominator}
 q(t)=1-a(t)+b(t)-E(t),\qquad p(z)=z^mq(1/z).
\end{equation}
Then $p$ is monic of degree $m$. The correction $E$ leaves the encoded even coefficients unchanged, so
\[
 [t^{2g+2+2d}]q=\frac{[t^d]S^2-\beta_d}{4}\quad\text{for }d\in\mathcal D.
\]

\begin{samepage}
\begin{lemma}\label{thm:encoding}
The constraints~\eqref{eq:squareproblem} are feasible if and only if $\dplus(1/p)=m$, where $p$ is defined in~\eqref{eq:denominator}.
\end{lemma}
\end{samepage}
\begin{proof}
Suppose $d_+(1/p)=m$. By Lemma~\ref{lem:compression}, there are polynomials $X,Y$ with nonnegative coefficients such that $X+Y=a$ and $[t^j]XY\ge[t^j]q$ for $2g\le j\le m$. Define $z_*=[t^g](X-Y)$ and $z_e=[t^{g+1+2e}](X-Y)$ for $0\le e\le D$, and set $Z(t)=\sum_{e=0}^D z_et^e$. The coefficient bounds $|[t^k](X-Y)|\le[t^k]a$ give
\begin{equation}\label{eq:dummy}
 X-Y=t^g\bigl(z_*+tZ(t^2)\bigr),\qquad |z_*|\le1,\quad |z_e|\le s_e.
\end{equation}
Since $X+Y=a$, we have
\[
 4XY=a^2-(X-Y)^2.
\]
The terms involving $z_*$ have degree $2g$ or odd degree. At the encoded degrees $j=2g+2+2d$ with $d\in\mathcal D$, the inequality $[t^j]XY\ge[t^j]q$ becomes
\[
 \frac{[t^d]S^2-[t^d]Z^2}{4}\ge
 \frac{[t^d]S^2-\beta_d}{4},
\]
which recovers all constraints on the selected coefficients of $Z^2$ from any positive realization of $1/p$ of order $m$.

Conversely, given $Z$ satisfying~\eqref{eq:squareproblem}, set
\begin{align}\label{eq:constructXY}
 X&=\tfrac12t^g\bigl(1+t(S+Z)(t^2)\bigr),\nonumber\\
 Y&=\tfrac12t^g\bigl(1+t(S-Z)(t^2)\bigr).
\end{align}
The polynomials $X,Y$ have nonnegative coefficients and satisfy $X+Y=a$. For each $d\in\mathcal D$, the coefficient of $XY$ at degree $2g+2+2d$ is at least that of $q$, while every other coefficient of $q$ of degree at least $2g$ is zero or negative. Lemma~\ref{lem:compression} therefore gives $d_+(1/p)=m$.
\end{proof}

\subsection{Positive realizability and stability}\label{sec:promise}

The odd coefficients supplied by $E$ ensure that $1/p$ has a positive realization whether or not the constraints in Definition~\ref{def:square} are feasible. A rational scaling then gives a stable transfer function without changing its minimum positive order.

\begin{lemma}\label{lem:existence}
For the polynomial $p$ in~\eqref{eq:denominator}, one can construct a positive rational number $K$ such that
\begin{equation}\label{eq:stable}
 \widehat H(z)=\frac{K^m}{p(Kz)}
\end{equation}
is stable, admits a rational positive realization, and satisfies $\dplus(\widehat H)=\dplus(1/p)$.
\end{lemma}
\begin{proof}
Set $s=\deg a$, $r=m+s$, and $L(z)=z^s(1+a(1/z))$. We realize $1/p=L/(pL)$ using a nonnegative companion matrix for $pL$. The coefficients of $a$ at degrees $g$ and $g+1$ are one. Since $E$ has coefficient $\gamma$ at every odd degree from $2g+1$ to $m$, the product $aE$ has coefficient at least $\gamma$ at every degree from $3g+1$ to $4g$. All nonzero coefficients of $ab$ occur between $3g+2$ and $3g+6D+3=4g-1$, and each is at most $a(1)b(1)<\gamma$. Hence $aE-ab\coeffge0$. Also,~\eqref{eq:boundpoly} gives
\[
 a^2-b=t^{2g}\bigl[1+2tS(t^2)+t^2(S^2-B)(t^2)\bigr]\coeffge0.
\]
It follows that
\[
 q(1+a)=1-\mathcal R,\qquad
 \mathcal R=(a^2-b)+E+(aE-ab)\coeffge0.
\]

The polynomial $L$ has nonnegative rational coefficients, and
\[
 P(z)=p(z)L(z)=z^r-\sum_{j=1}^rR_jz^{r-j},\qquad R_j\ge0,
\]
where $\mathcal R(t)=\sum_{j=1}^rR_jt^j$. Let $A_\star$ be the row companion matrix with unit superdiagonal, last row $(R_r,\ldots,R_1)$, and all other entries zero. Let $\ell\in\R^r$ contain the coefficients of $L$ in increasing powers of $z$, padded with zeros. The identity
\[
 (zI-A_\star)^{-1}\e_r
 =\frac{(1,z,\ldots,z^{r-1})^\top}{P(z)}
\]
gives $\ell^\top(zI-A_\star)^{-1}\e_r=L/P=1/p$. Thus $(A_\star,\e_r,\ell)$ is a rational positive realization.

Choose $K=2(1+\sum_{j=1}^r R_j)$. Every row sum of $A_\star/K$ is less than $1/2$, so the scaled realization $(A_\star/K,\e_r,K^{m-1}\ell)$ is stable and realizes~\eqref{eq:stable}. More generally, the map
\[
 (A,b,c)\longmapsto(A/K,\,b,\,K^{m-1}c)
\]
is a bijection between positive realizations of $1/p$ and $\widehat H$ that preserves dimension. Their minimum positive orders are therefore equal.
\end{proof}

\section{Proof of Theorem~\ref{thm:main}}\label{sec:bits}

Section~\ref{sec:construction} establishes the equivalence between $\Inv$ feasibility and positive realization at the McMillan degree. To prove Theorem~\ref{thm:main}, we express positive realization within a prescribed order bound as real polynomial feasibility and verify that the construction takes polynomial time.

\subsection{Membership in \texorpdfstring{$\ER$}{the existential real complexity class}}
A finite set of Markov equations suffices to certify exact equality of transfer matrices.
\begin{lemma}\label{lem:membership}
$\PR\in\ER$.
\end{lemma}
\begin{proof}
Write each entry of $H$ in reduced form and let $d$ be the maximum denominator degree, taking denominator one for zero entries. For $p$ inputs and $q$ outputs, introduce nonnegative variables $A\in\R^{n\times n}$, $C\in\R^{q\times n}$, and $X_0,\ldots,X_{d+n-1}\in\R^{n\times p}$. Use $X_0$ as the input matrix $B$ and $X_k$ to represent $A^kB$. The auxiliary variables let us express the Markov conditions through quadratic equations. Impose
\begin{align}\label{eq:finite}
 X_{k+1}&=AX_k &&\text{for }0\le k<d+n-1,\nonumber\\
 CX_k&=H_k &&\text{for }0\le k<d+n.
\end{align}
The recurrence gives $X_k=A^kB$, so~\eqref{eq:finite} matches the first $d+n$ Markov matrices. Each entry of $H-C(zI-A)^{-1}B$ is therefore $O(z^{-(d+n+1)})$ and has denominator degree at most $d+n$. Multiplying by its denominator gives a polynomial numerator that tends to zero as $z\to\infty$. This polynomial must be identically zero, so the transfer matrices agree exactly. Conversely, any positive realization with at most $n$ states can be padded with zero states to satisfy~\eqref{eq:finite}.

The full coefficient lists and unary encoding of $n$ bound $d$ and $n$ by the input length. Hence~\eqref{eq:finite} contains polynomially many scalar variables and quadratic equations. The required rational Markov parameters are computable from the coefficient recurrence for each entry of $H$ in polynomial time and have polynomial bit length.
\end{proof}

\subsection{Polynomial size and completion of the reduction}\label{sec:mainproof}
\begin{proof}[Proof of Theorem~\ref{thm:main}]
Let $\Phi$ be an $\Inv$ instance with encoding length $N_0$. Lemma~\ref{lem:squarehard} constructs the coefficient data in polynomial time, with $D$ bounded by a polynomial in $N_0$. By~\eqref{eq:sizes}, the integers $g,m$ and the companion dimension $r=m+\deg a$ in Lemma~\ref{lem:existence} are linear in $D$. Thus the unary order bound and all coefficient lists have polynomial length.

The rational coefficients of $S,B$ have polynomial total bit length. The fixed number of polynomial sums and products defining $a,b,\gamma,E,q$ and $\mathcal R$ preserves this bound. The scale $K=2(1+\sum_jR_j)$ therefore has polynomial bit length, as do its powers up to $m$. Consequently, the coefficients of $\widehat H$ are computable in polynomial time with polynomial bit length.

Write $\widehat H_\Phi$ for the transfer function~\eqref{eq:stable} constructed from $\Phi$, and $m_\Phi$ for its McMillan degree. Lemmas~\ref{lem:squarehard}, \ref{thm:encoding}, and~\ref{lem:existence} show that $\Phi$ is feasible if and only if
\begin{equation}\label{eq:finalequiv}
 \dplus(\widehat H_\Phi)\le m_\Phi=\mcm(\widehat H_\Phi).
\end{equation}
Hence $\Phi\mapsto(\widehat H_\Phi,m_\Phi)$ is a polynomial-time reduction from $\Inv$ to $\PR$. Lemma~\ref{lem:existence} ensures that every $\widehat H_\Phi$ is stable and admits a positive realization. Since $\Inv$ is $\ER$-complete, the reduction proves hardness with the restrictions in Theorem~\ref{thm:main}. Lemma~\ref{lem:membership} establishes membership in $\ER$.
\end{proof}

\section{Conclusion}\label{sec:conclusion}
We have established that deciding whether a positive realization exists within a prescribed order bound is complete for the existential theory of the reals. Hardness holds even for stable SISO transfer functions that admit positive realizations, with the bound equal to the McMillan degree. Thus, determining whether nonnegativity requires additional states is already computationally hard, even when stability and positive realizability are guaranteed.

\appendix
\section{Proofs of Lemmas~\ref{lem:hess} and~\ref{lem:compression}}\label{app:structure}

\begin{proof}[Proof of Lemma~\ref{lem:hess}]
By~\eqref{eq:prefix}, the first nonzero Markov parameter is $h_{m-1}=1$.
A nonzero Markov parameter is a sum of weights of positive input--output walks. Thus a shortest positive input--output path has $m-1$ edges. A shortest path repeats no vertex, so in an $m$-state realization it visits all states. Label them in its order. Input support at a later state, or output support at an earlier state, would shorten a positive path. Hence $b=\beta\e_1$ and $c=\gamma\e_m$ with $\beta,\gamma>0$. An edge $j\to i$ with $i\ge j+2$ would likewise create a path of length at most $m-2$. Therefore $A$ is upper Hessenberg, with subdiagonal entries $t_i=A_{i+1,i}>0$.

A walk of length $m-1$ from state one to state $m$ must use the subdiagonal at every step. Equation~\eqref{eq:prefix} gives
\begin{equation}\label{eq:pathprod}
 \beta\gamma\prod_{i=1}^{m-1}t_i=1.
\end{equation}
Define $d_1=\beta$, $d_{i+1}=t_id_i$, and $D=\diag(d_i)$. Then $J=D^{-1}AD$ has unit subdiagonal, $D^{-1}b=\e_1$, and $c^\top D=\e_m^\top$. Since the reduced denominator already has degree $m$, $\det(zI-J)=p$.

Conversely, Cramer's rule and the unit subdiagonal in~\eqref{eq:hess} give
\begin{equation}\label{eq:cofactor}
 \e_m^\top(zI-J)^{-1}\e_1=\frac1{\det(zI-J)}=\frac1p.
\end{equation}
The McMillan degree gives the lower bound $m$, proving minimality.
\end{proof}

Similarity transformations to nonnegative Hessenberg form in the general matrix setting are studied in \cite{GrusslerRantzer2022}. Here the impulse-response condition~\eqref{eq:prefix} determines the state structure and both endpoint vectors.

\begin{proof}[Proof of Lemma~\ref{lem:compression}]

The normalized matrix in~\eqref{eq:hess} is the transpose of the lower Hessenberg form with unit superdiagonal in Torre-Mayo et al.~\cite[Definition~10]{TorreMayo2007}. Its characteristic polynomial admits the expansion over vertex-disjoint directed cycles in \cite[Eq.~(5)]{TorreMayo2007}. Lemma~\ref{lem:hess} shows that the initial zero Markov parameters force every positive realization of order $m$ into this matrix form with input and output vectors $\e_1$ and $\e_m$.

Write $w_{ij}=J_{ij}$ for $1\le i\le j\le m$. The edge $j\to i$, together with the directed path of unit subdiagonal edges, is a directed cycle on the interval $[i,j]$, of length $j-i+1$ and weight $w_{ij}$. Every simple directed cycle has this form. When a cycle is traversed from its smallest vertex, the first backward step must return to that vertex, since returning to any other earlier vertex would repeat a vertex before closing the cycle.

The determinant expansion gives

\begin{equation}\label{eq:interval}
 q(t)=\sum_{\mathcal I}(-1)^{|\mathcal I|}
 t^{\sum_{I\in\mathcal I}|I|}\prod_{I\in\mathcal I}w_I,
\end{equation}

where $\mathcal I$ ranges over families of pairwise vertex-disjoint intervals, including the empty family, and $w_I=w_{ij}$ for $I=[i,j]$.

Torre-Mayo et al.~\cite[Lemma~37]{TorreMayo2007} aggregate weights of cycles of a common length to bound the coefficient associated with pairs of such cycles. Under the size and coefficient assumptions of Lemma~\ref{lem:compression}, the compression applies to several cycle lengths. We adjust the weights of long cycles to compensate for the additional pair products and preserve both the characteristic polynomial and the transfer function $\e_m^\top(zI-J)^{-1}\e_1$.

By Lemma~\ref{lem:hess}, it suffices to prove the equivalent existence statement for a nonnegative upper Hessenberg matrix with unit subdiagonal and reversed characteristic polynomial $q$.
The conditions $q_1=\cdots=q_{g-1}=0$ imply inductively that every cycle of length less than $g$ has zero weight. For $g\le k<2g$, products of two nonzero cycle weights cannot contribute at degree $k$. Hence $q_k$ is minus the total weight of length-$k$ cycles, and that total weight equals $s_k$.

Since $m<3g$, at most two nonzero cycles can be disjoint, and a cycle of length at least $2g$ cannot be disjoint from another nonzero cycle. Hence, for $2g\le j\le m$, the interval expansion gives
\begin{equation}\label{eq:surplus}
 q_j=Q_j-u_j,
\end{equation}
where short cycles have lengths in $[g,2g-1]$, $Q_j$ is the sum of products of weights of disjoint short cycles with total length $j$, and $u_j\ge0$ is the total weight of length-$j$ cycles.

Partition the short intervals according to whether their starting vertex is at most $g$. The first class shares vertex $g$. An interval in the second class starts no later than $m-g+1=2g$ and ends no earlier than $2g$, so this class shares vertex $2g$. Let $X_k,Y_k$ be the total weights of length-$k$ intervals in the first and second classes, respectively. Every disjoint pair consists of one interval from each class and is counted once, so $X_k+Y_k=s_k$ and
\begin{equation}\label{eq:monotonicity}
 q_j\le Q_j\le[t^j]XY\qquad\text{for }2g\le j\le m.
\end{equation}

Conversely, place weight $X_k$ on $[1,k]$ and weight $Y_k$ on $[m-k+1,m]$. A prefix and suffix of lengths $k,\ell$ are disjoint exactly when $k+\ell\le m$. Thus their contribution at each degree $j\le m$ is $[t^j]XY$. For $2g\le j\le m$, add the cycle on $[1,j]$ with weight
\begin{equation}\label{eq:longweight}
 u_j=[t^j]XY-q_j\ge0.
\end{equation}
The correction by a long cycle is illustrated in Fig.~\ref{fig:stateexample}. These long intervals cannot participate in disjoint pairs. The interval expansion therefore matches every coefficient of $q$. Terms of $XY$ above degree $m$ correspond to overlapping intervals and do not contribute. The transfer function $\e_m^\top(zI-J)^{-1}\e_1$ is preserved.
\end{proof}

\section{Proof of Lemma~\ref{lem:squarehard}}\label{app:encoding}
Apply the preprocessing in Section~\ref{sec:source}, obtaining variables $x_0,\ldots,x_{n-1}$ with $n\ge1$. Let $v_i$ represent $x_i/2$ and $v_{n+i}$ its negative for $0\le i<n$. These variables lie in $[-1,1]$, and their linear conditions are
\[
 -v_i\le-1/4,\qquad v_i+v_{n+i}=0\quad\text{for }0\le i<n,
\]
together with $v_i+v_j-v_k=0$ for each addition equation. Replace equalities by two inequalities, combine repeated indices, and remove zero identities. Write the resulting $Q\ge3n$ nonzero rows as $\sum_i a_{ji}v_i\le b_j$, with $0\le j<Q$. Each row has at most three nonzero coefficients, $|a_{ji}|\le2$, and $b_j\in\{0,-1/4\}$.

We choose powers for the variable coefficients and two coefficients of opposite unit signs so that distinct unordered pairs contribute at distinct degrees in $Z^2$. Set $h=2n+2$ and
\[
 e_k=2hk^2+k\quad\text{for }0\le k<h,\qquad U=e_{h-1}.
\]
Equality of two pair sums gives
\[
 2h(i^2+k^2-p^2-q^2)=p+q-i-k.
\]
The right side has absolute value less than $2h$, so both sides vanish. The pairs have equal sums and sums of squares, hence equal products and the same indices. Use $e_0=0$ and $e_1$ for the opposite sign coefficients. Set $f_i=e_{i+2}$ and write $u_i=z_{f_i}$ for $0\le i<2n$.

To encode each nonzero row coefficient $a_{ji}$, place an auxiliary coefficient at $\alpha_{ji}=T_j-f_i$, where
\[
 w=2U+1,\qquad T_j=(Q+j+1)w,\qquad L=2Qw.
\]
Products with $u_i$ contribute at degree $T_j$. The auxiliary positions in row $j$ lie in $[T_j-U,T_j]$. These disjoint intervals lie above $2U$, and products of two auxiliary coefficients have degree greater than $L$. Define $S$ by
\[
 s_{e_k}=1,\qquad s_{\alpha_{ji}}=|a_{ji}|/2,
\]
with all other coefficients zero, and let $D=\max_{a_{ji}\ne0}\alpha_{ji}$. Then $s_0=1$, $s_D>0$, and $0\le s_e\le1$. Impose $|z_e|\le s_e$ throughout. The constraint
\[
 [t^{e_1}]Z^2\le-2
\]
is $2z_{e_0}z_{e_1}\le-2$, so the unit bounds force $z_{e_0}=\varepsilon$ and $z_{e_1}=-\varepsilon$ for some $\varepsilon\in\{-1,1\}$. Set $v_i=\varepsilon u_i$ to correct this common sign.

For $a_{ji}>0$, impose $[t^{\alpha_{ji}+e_1}]Z^2\le-|a_{ji}|$. For $a_{ji}<0$, impose this bound at $\alpha_{ji}+e_0$. Each selected coefficient contains only the indicated product of an auxiliary and a reference coefficient. To see this, its degree lies between $T_j-U$ and $T_j$, above every sum of two exponents from $\{e_0,\ldots,e_{h-1}\}$ and below every sum of two auxiliary exponents. A competing product would satisfy
\[
 e_p+T_k-f_\ell=T_j-f_i+e_\rho,\qquad \rho\in\{0,1\}.
\]
Since $|e_p+f_i-f_\ell-e_\rho|\le2U<w$, we have $k=j$. Uniqueness of pair sums then forces $p=\rho$ and $\ell=i$, because no $f_i$ is a reference position. The imposed bound therefore forces $z_{\alpha_{ji}}=\varepsilon a_{ji}/2$.

At degree $T_j$, the separation isolates row $j$:
\[
 [t^{T_j}]Z^2=2\sum_{i:a_{ji}\ne0}u_i z_{\alpha_{ji}}
 =\sum_i a_{ji}v_i.
\]
Imposing $[t^{T_j}]Z^2\le b_j$ expresses row $j$. Thus $1/4\le v_i\le1$, $v_{n+i}=-v_i$ for $0\le i<n$, and the decoded variables $x_i=2v_i$ satisfy the addition equations.

For each equation $x_ix_j=1$, impose $[t^{f_i+f_j}]Z^2\le1/2$ if $i\ne j$, or $[t^{2f_i}]Z^2\le1/4$ if $i=j$. Also impose $[t^{f_i+f_{n+j}}]Z^2\le-1/2$. These degrees are at most $2U$, and uniqueness of pair sums isolates the required products. The bounds give $u_iu_j\le1/4$ and $u_iu_{n+j}\le-1/4$. Since $u_{n+j}=-u_j$, they force $v_iv_j=1/4$, hence $x_ix_j=1$. Conversely, a solution $x$ gives feasible free coefficients $u=(x,-x)/2$, reference coefficients $z_{e_0}=1$, $z_{e_1}=-1$, and auxiliary coefficients $a_{ji}/2$.

Let $\mathcal D$ contain the imposed degrees, retaining the smallest bound if a degree occurs more than once. Every selected degree is at most $L$. The last row is nonzero, so $D\ge L-U$ and $2D>L$. Hence $\mathcal D\subseteq\{0,\ldots,2D\}$. The reference and auxiliary bounds equal the negatives of the corresponding coefficients of $S^2$. At degree $T_j$, that coefficient is $\sum_i|a_{ji}|\ge1$, while $b_j\in\{0,-1/4\}$. At the inversion degrees, the coefficients of $S^2$ are respectively $2$, $1$, and $2$ for the bounds $1/2$, $1/4$, and $-1/2$. Thus every bound satisfies~\eqref{eq:range}.

Finally, $U<2h^3$ and $D<L=2Q(2U+1)=O(Qh^3)$. Since $Q$ is linear in the number of variables and equations and all coefficients have bounded bit length, the full coefficient list for $S$, the degree set, and the bounds are computable in polynomial time. \qed

\bibliographystyle{IEEEtran}
\bibliography{references}

@book{FarinaRinaldi2000,
  author={Lorenzo Farina and Sergio Rinaldi},
  title={Positive Linear Systems: Theory and Applications},
  publisher={John Wiley \& Sons}, address={New York}, year={2000},
  doi={10.1002/9781118033029}
}

@article{BenvenutiFarina1999,
  author={Luca Benvenuti and Lorenzo Farina},
  title={An example of how positivity may force realizations of `large' dimension},
  journal={Systems \& Control Letters},
  volume={36}, number={4}, pages={261--266}, year={1999},
  doi={10.1016/S0167-6911(98)00098-X}
}

@article{BenvenutiFarina2004,
  author={Luca Benvenuti and Lorenzo Farina},
  title={A tutorial on the positive realization problem},
  journal={IEEE Transactions on Automatic Control},
  volume={49}, number={5}, pages={651--664}, year={2004},
  doi={10.1109/TAC.2004.826715}
}

@article{Benvenuti2022,
  author={Luca Benvenuti},
  title={Minimal positive realizations: A survey},
  journal={Automatica}, volume={143}, pages={110422}, year={2022},
  doi={10.1016/j.automatica.2022.110422}
}

@inproceedings{TaghavianSjolund2025,
  author={Hamed Taghavian and Jens Sj{\"o}lund},
  title={Minimal positive {Markov} realizations},
  booktitle={2025 IEEE 64th Conference on Decision and Control (CDC)},
  pages={858--863}, year={2025}, publisher={IEEE},
  doi={10.1109/CDC57313.2025.11312247}
}

@article{BorobiaCanogar2017,
  author={Alberto Borobia and Roberto Canogar},
  title={The real nonnegative inverse eigenvalue problem is {NP}-hard},
  journal={Linear Algebra and its Applications},
  volume={522}, pages={127--139}, year={2017},
  doi={10.1016/j.laa.2017.02.010}
}

@misc{Shitov2018,
  author={Yaroslav Shitov},
  title={A universality theorem for nonnegative matrix factorizations},
  year={2018},
  note={{arXiv:1606.09068v2}}
}

@misc{Abrahamsen2018,
  author={Mikkel Abrahamsen and Anna Adamaszek and Tillmann Miltzow},
  title={The art gallery problem is {$\exists\mathbb{R}$}-complete},
  year={2018},
  note={{arXiv:1704.06969v4}}
}

@article{GrusslerRantzer2022,
  author={Christian Grussler and Anders Rantzer},
  title={On the similarity to nonnegative and {Metzler Hessenberg} forms},
  journal={Special Matrices}, volume={10}, number={1}, pages={1--8}, year={2022},
  doi={10.1515/spma-2020-0140}
}

@article{TorreMayo2007,
  author={J. Torre-Mayo and M. R. Abril-Raymundo and E. Alarcia-Est{\'e}vez and C. Mariju{\'a}n and M. Pisonero},
  title={The nonnegative inverse eigenvalue problem from the coefficients of the characteristic polynomial. {EBL} digraphs},
  journal={Linear Algebra and its Applications},
  volume={426}, number={2--3}, pages={729--773}, year={2007},
  doi={10.1016/j.laa.2007.06.014}
}

@article{Cortese2025,
  author={Marco Cortese and Tommaso Grigoletto and Francesco Ticozzi and Augusto Ferrante},
  title={Robust, positive, and exact model reduction via monotone matrices},
  journal={IEEE Transactions on Automatic Control},
  volume={70}, number={10}, pages={6907--6912}, year={2025},
  doi={10.1109/TAC.2025.3569190}
}

@article{CzajaJamingMatolcsi2008,
  author={Wojciech Czaja and Philippe Jaming and M{\'a}t{\'e} Matolcsi},
  title={An efficient algorithm for positive realizations},
  journal={Systems \& Control Letters},
  volume={57}, number={5}, pages={436--441}, year={2008},
  doi={10.1016/j.sysconle.2007.11.001}
}

@article{HalmschlagerMatolcsi2005,
  author={Andrea Halmschlager and M{\'a}t{\'e} Matolcsi},
  title={Minimal positive realizations for a class of transfer functions},
  journal={IEEE Transactions on Circuits and Systems II: Express Briefs},
  volume={52}, number={4}, pages={177--180}, year={2005},
  doi={10.1109/TCSII.2004.842420}
}

@article{NagyMatolcsi2003Lowerbound,
  author={B{\'e}la Nagy and M{\'a}t{\'e} Matolcsi},
  title={A lowerbound on the dimension of positive realizations},
  journal={IEEE Transactions on Circuits and Systems I: Fundamental Theory and Applications},
  volume={50}, number={6}, pages={782--784}, year={2003},
  doi={10.1109/TCSI.2003.812609}
}

@article{VanDenHof1997,
  author={van den Hof, J. M.},
  title={Realization of positive linear systems},
  journal={Linear Algebra and its Applications},
  volume={256}, pages={287--308}, year={1997},
  doi={10.1016/S0024-3795(96)00030-4}
}

@article{Vavasis2009,
  author={Stephen A. Vavasis},
  title={On the complexity of nonnegative matrix factorization},
  journal={SIAM Journal on Optimization},
  volume={20}, number={3}, pages={1364--1377}, year={2009},
  doi={10.1137/070709967}
}

@article{Benvenuti2020NIEP,
  author={Luca Benvenuti},
  title={The {NIEP} and the positive realization problem},
  journal={Electronic Journal of Linear Algebra},
  volume={36}, pages={367--384}, year={2020},
  doi={10.13001/ela.2020.5039}
}

@misc{SchaeferCardinalMiltzow2024,
  author={Marcus Schaefer and Jean Cardinal and Tillmann Miltzow},
  title={The Existential Theory of the Reals as a Complexity Class: A Compendium},
  year={2024},
  doi={10.48550/arXiv.2407.18006},
  note={{arXiv:2407.18006}}
}

@article{Benvenuti2013RealPoles,
  author={Luca Benvenuti},
  title={Minimal positive realizations of transfer functions with real poles},
  journal={IEEE Transactions on Automatic Control},
  volume={58}, number={4}, pages={1013--1017}, year={2013},
  doi={10.1109/TAC.2012.2212612}
}

@article{Hadjicostis1999,
  author={Christoforos Hadjicostis},
  title={Bounds on the size of minimal nonnegative realizations for discrete-time {LTI} systems},
  journal={Systems \& Control Letters},
  volume={37}, number={1}, pages={39--43}, year={1999},
  doi={10.1016/S0167-6911(99)00005-5}
}
\end{document}